\documentclass[12pt,reqno]{amsart}
\usepackage{amsmath,amssymb,amsthm}

\calclayout

\theoremstyle{plain}
\newtheorem{thm}{Theorem}[section]
\newtheorem{theorem}[thm]{Theorem}

\newtheorem{corollary}[thm]{Corollary}

\theoremstyle{definition}

\newtheorem{definition}[thm]{Definition}

\newtheorem{rem}[thm]{Remark}

\title{A new proof of the duality of multiple zeta values and its generalizations}
\author{Shin-ichiro Seki and Shuji Yamamoto}
\address{Mathematical Institute, Tohoku University, 6-3, Aoba, Aramaki, Aoba-Ku, 
Sendai, 980-8578, Japan}
\email{shinichiro.seki.b3@tohoku.ac.jp}
\address{Keio Institute of Pure and Applied Sciences (KiPAS), Graduate School of
Science and Technology, Keio University, 3-14-1 Hiyoshi, Kohoku-ku, 
Yokohama, 223-8522, Japan}
\email{yamashu@math.keio.ac.jp}
\thanks{This work was supported in part by JSPS KAKENHI Grant Numbers 
JP18J00151, JP16H06336, JP16K13742, JP18K03221, 
as well as the KiPAS program 2013--2018 of the Faculty of Science and Technology 
at Keio University. }

\begin{document}
	
\begin{abstract}
We give a new proof of the duality of multiple zeta values, 
which makes no use of the iterated integrals. 
The same method is also applicable to Ohno's relation 
for ($q$-)multiple zeta values. 
\end{abstract}
	
\maketitle

\section{Introduction}
We call a tuple of positive integers $\boldsymbol{k}=(k_1, \dots, k_r)$ 
an \emph{index} and 
define its \emph{weight} $\mathrm{wt}(\boldsymbol{k})$ to be $k_1+\cdots+k_r$. 
An index $\boldsymbol{k}$ is called \emph{admissible} when $k_r \geq 2$. 
For an admissible index $\boldsymbol{k}$, the \emph{multiple zeta value} 
$\zeta(\boldsymbol{k})$ is defined as a convergent series
\[
\zeta(\boldsymbol{k}):=\sum_{0 < m_1 < \cdots < m_r}\frac{1}{m_1^{k_1}\cdots m_r^{k_r}}.
\]
Given an admissible index $\boldsymbol{k}$, we write it in the form 
\[
\boldsymbol{k}=(\{1\}^{a_1-1}, b_1+1, \dots, \{1\}^{a_s-1}, b_s+1), 
\]
where $a_1, \dots, a_s, b_1, \dots, b_s$ are positive integers 
and $\{1\}^a$ means $1,\ldots,1$ repeated $a$ times, 
and then define its \emph{dual index} $\boldsymbol{k}^{\dagger}$ by
\[
\boldsymbol{k}^{\dagger}=(\{1\}^{b_s-1}, a_s+1, \dots, \{1\}^{b_1-1}, a_1+1). 
\]
The following relation is one of the most famous relations among 
multiple zeta values.

\begin{theorem}[Duality]
Let $\boldsymbol{k}$ be an admissible index and $\boldsymbol{k}^{\dagger}$ its dual. 
Then we have
\[
\zeta(\boldsymbol{k})=\zeta(\boldsymbol{k}^{\dagger}).
\]
\end{theorem}

The duality, conjectured by Hoffman \cite{H}, 
was proved by making change of variables in the iterated integral representation 
of the multiple zeta value (see Zagier \cite{Z}). 
In this note, we give a new proof of the duality without using iterated integrals. 
The proof is described briefly as follows. 
We introduce the \emph{connected sum} 
\begin{equation}\label{eq:Z(k;l)}
Z(\boldsymbol{k}; \boldsymbol{l}):=
\sum_{\substack{0=m_0<m_1<\cdots<m_r \\ 0=n_0<n_1<\cdots<n_s}}
\prod_{i=1}^r\frac{1}{m_i^{k_i}}\prod_{j=1}^s\frac{1}{n_j^{l_j}}
\cdot\frac{m_r!\cdot n_s!}{(m_r+n_s)!}
\end{equation}
for indices $\boldsymbol{k}=(k_1, \dots, k_r)$ and $\boldsymbol{l}=(l_1, \dots, l_s)$ 
(the name ``connected sum'' means that the two series 
\[\sum_{0=m_0<m_1<\cdots<m_r}\prod_{i=1}^r\frac{1}{m_i^{k_i}}
\quad \text{and} \quad 
\sum_{0=n_0<n_1<\cdots<n_s}\prod_{j=1}^s\frac{1}{n_j^{l_j}}
\]
are connected by the factor $\frac{m_r!\cdot n_s!}{(m_r+n_s)!}$; 
it is not related to the notion of the same name in topology). 
Then, by using an identity
\[
\sum_{a=m+1}^{\infty}\frac{1}{a}\cdot\frac{a!\cdot n!}{(a+n)!}
=\frac{1}{n}\cdot\frac{m!\cdot n!}{(m+n)!}
\]
$\mathrm{wt}(\boldsymbol{k})$ times,
we see that $\zeta(\boldsymbol{k})=Z(\boldsymbol{k}; \varnothing)$ is equal to 
$Z(\varnothing; \boldsymbol{k}^{\dagger})=\zeta(\boldsymbol{k}^{\dagger})$, 
where $a$ and $n$ are positive integers and $m$ is a non-negative integer.
For example, Euler's relation $\zeta(1, 2)=\zeta(3)$ which is the simplest 
non-trivial case of the duality is proved through the following steps:
\begin{alignat*}{3}
\zeta(1,2)&=\sum_{0<m_1<m_2}\frac{1}{m_1m_2^2}&\quad &=Z(1,2; \varnothing)\\
&=\sum_{0<m_1<m_2}\sum_{n=1}^{\infty}\frac{1}{m_1m_2}\cdot\frac{1}{n}\cdot\frac{m_2!\cdot n!}{(m_2+n)!}&\quad &=Z(1,1; 1) \\
&=\sum_{m_1=1}^{\infty}\sum_{n=1}^{\infty}\frac{1}{m_1}\cdot\frac{1}{n^2}\cdot\frac{m_1!\cdot n!}{(m_1+n)!}&\quad &=Z(1; 2) \\
&=\sum_{n=1}^{\infty}\frac{1}{n^3}&\quad &=Z(\varnothing; 3)\quad =\zeta(3).
\end{alignat*}
By the same argument, we can obtain new proofs of Ohno's relation \cite{O} 
and its $q$-analog due to Bradley \cite{B}. 
To recall the statements, we introduce some notation. 
Let $q$ be a real number satisfying $0<q<1$. 
For an admissible index $\boldsymbol{k}=(k_1, \dots, k_r)$, we define the \emph{$q$-multiple zeta value} 
$\zeta_q(\boldsymbol{k})$ by
\[
\zeta_q(\boldsymbol{k}):=\sum_{0 < m_1 < \cdots < m_r}
\frac{q^{(k_1-1)m_1+\cdots+(k_r-1)m_r}}{[m_1]_q^{k_1}\cdots [m_r]_q^{k_r}},
\]
where the $q$-integer $[m]_q$ is defined to be $\frac{1-q^m}{1-q}$ 
for a positive integer $m$. 
Then we have $\lim_{q \to 1}\zeta_q(\boldsymbol{k})=\zeta(\boldsymbol{k})$. 
For a non-negative integer $c$ and an admissible index $\boldsymbol{k}=(k_1, \dots, k_r)$, 
we consider a linear combination of $q$-multiple zeta values 
\[
S_q(\boldsymbol{k};c) := 
\sum_{\substack{c_1+\cdots+c_r=c \\ c_i \geq 0}}\zeta_q(k_1+c_1, \dots, k_r+c_r).
\]
Then the $q$-analog of Ohno's relation is the following:
\begin{theorem}[{Bradley \cite[Theorem 5]{B}}]\label{mainthm}
Let $c$ be a non-negative integer, 
$\boldsymbol{k}$ an admissible index and 
$\boldsymbol{k}^{\dagger}$ its dual. 
Then 
\[
S_q(\boldsymbol{k};c) = S_q(\boldsymbol{k}^{\dagger};c).
\]
\end{theorem}
By taking the limit $q \to 1$, one obtains Ohno's relation. 
The duality is the case $c=0$ of Ohno's relation.

\section{New proof of Theorem \ref{mainthm}}
\begin{definition}[Connected sum]
Let $\boldsymbol{k}=(k_1, \dots, k_r)$ and $\boldsymbol{l}=(l_1, \dots, l_s)$ be 
(possibly non-admissible or empty) indices and let $x$ be a real number satisfying $\lvert x\rvert<1$. 
We define a series $Z_q(\boldsymbol{k}; \boldsymbol{l}; x)$ by 
\begin{align*}
Z_q(\boldsymbol{k}; \boldsymbol{l}; x)
:=\sum_{\substack{0=m_0<m_1<\cdots<m_r \\ 0=n_0<n_1<\cdots<n_s}}
&\prod_{i=1}^r\frac{q^{(k_i-1)m_i}}{([m_i]_q-q^{m_i}x)[m_i]_q^{k_i-1}}
\prod_{j=1}^s\frac{q^{(l_j-1)n_j}}{([n_j]_q-q^{n_j}x)[n_j]_q^{l_j-1}}\\ 
&\cdot\frac{q^{m_rn_s}f_q(m_r;x)f_q(n_s;x)}{f_q(m_r+n_s;x)},
\end{align*}
where $f_q(m;x)=\prod_{h=1}^m([h]_q-q^hx)$ for a non-negative integer $m$. 
Since this is a non-empty series of positive terms, 
its value is well-defined as a positive real number (if convergent) or 
as positive infinity $+\infty$ (if divergent), independently of the order of the terms. 
\end{definition}

Note that there is a symmetry 
$Z_q(\boldsymbol{k}; \boldsymbol{l};x)=Z_q(\boldsymbol{l}; \boldsymbol{k};x)$. 
We also remark that, letting $q\to 1$ and $x\to 0$, 
we recover the series $Z(\boldsymbol{k};\boldsymbol{l})$ defined in \eqref{eq:Z(k;l)}. 

\begin{theorem}\label{key-thm}
Let $r$ and $s$ be non-negative integers and $k_1, \dots, k_r$, $l_1, \dots, l_s$ 
positive integers. If $s>0$, then we have
\[
Z_q(k_1, \dots, k_r, 1; l_1, \dots, l_s;x)=Z_q(k_1, \dots, k_r; l_1, \dots, l_s+1;x), 
\]
and if $r>0$, then we have
\[
Z_q(k_1, \dots, k_r+1; l_1, \dots, l_s;x)=Z_q(k_1, \dots, k_r; l_1, \dots, l_s, 1;x).
\]
\end{theorem}
\begin{proof}
The first equality follows from the telescoping sum 
\begin{align*}
&\sum_{a>m}\frac{1}{[a]_q-q^ax}\cdot\frac{q^{an}f_q(a;x)f_q(n;x)}{f_q(a+n;x)}\\
&=\frac{q^n}{[n]_q}\sum_{a>m}
\Biggl(\frac{q^{(a-1)n}f_q(a-1;x)f_q(n;x)}{f_q(a-1+n;x)}
-\frac{q^{an}f_q(a;x)f_q(n;x)}{f_q(a+n;x)}\Biggr) \\
&=\frac{q^n}{[n]_q}\cdot\frac{q^{mn}f_q(m;x)f_q(n;x)}{f_q(m+n;x)}, 
\end{align*}
applied to $m=m_r$, $n=n_s$ and $a=m_{r+1}$ in the definition of 
$Z_q(k_1, \dots, k_r, 1; l_1, \dots, l_s;x)$. 
The second equality is equivalent to the first by the symmetry of $Z_q$. 
\end{proof}

\begin{corollary}
For an admissible index $\boldsymbol{k}$, put 
\[
Z_q(\boldsymbol{k}; x):=\sum_{0<m_1<\cdots<m_r}\prod_{i=1}^r\frac{q^{(k_i-1)m_i}}{([m_i]_q-q^{m_i}x)[m_i]_q^{k_i-1}}.
\]
This is convergent, and we have
\begin{equation}
Z_q(\boldsymbol{k}; x)=Z_q(\boldsymbol{k}^{\dagger}; x).
\label{x-duality}\end{equation}
\end{corollary}
\begin{proof}
When $0<x<1$, we have $\frac{1}{[m]_q-q^m x}\leq\frac{1}{1-x}\cdot\frac{1}{[m]_q}$ and hence 
\[Z_q(\boldsymbol{k}; x)\leq \biggl(\frac{1}{1-x}\biggr)^r \zeta_q(\boldsymbol{k})<+\infty. \]
When $x\leq 0$, we readily see that $Z_q(\boldsymbol{k}; x) \leq \zeta_q(\boldsymbol{k})<+\infty$. 
Thus $Z_q(\boldsymbol{k}; x)$ is convergent. 
Moreover, 
by applying equalities in Theorem \ref{key-thm} $\mathrm{wt}(\boldsymbol{k})$ times, 
we see that 
\[
Z_q(\boldsymbol{k};x)
=Z_q(\boldsymbol{k}; \varnothing;x)=\cdots
=Z_q(\varnothing; \boldsymbol{k}^{\dagger};x)
=Z_q(\boldsymbol{k}^{\dagger};x)
\]
holds by the definitions of $Z_q(\boldsymbol{k}; x)$ and the dual index. 
\end{proof}

\begin{proof}[Proof of Theorem {\upshape\ref{mainthm}}]
Let us expand $Z_q(\boldsymbol{k};x)$ into a power series 
with respect to $x$. 
First, fix a sequence $0<m_1<\cdots<m_r$ and $i=1,\ldots,r$. Then one has 
\[
\frac{q^{(k_i-1)m_i}}{([m_i]_q-q^{m_i}x)[m_i]_q^{k_i-1}}=q^{(k_i-1)m_i}\sum_{c_i=0}^\infty\frac{q^{c_im_i}x^{c_i}}{[m_i]_q^{c_i+1}}\cdot\frac{1}{[m_i]_q^{k_i-1}}=\sum_{c_i=0}^{\infty}\frac{q^{(k_i+c_i-1)m_i}}{[m_i]_q^{k_i+c_i}}x^{c_i}.
\]
Taking the sum of products $\sum_{0<m_1<\cdots<m_r}\prod_{i=1}^r$ 
of this equality, we see that 
\[Z_q(\boldsymbol{k};x)=\sum_{c=0}^\infty 
S_q(\boldsymbol{k};c)x^c. \]
Hence, we get $S_q(\boldsymbol{k};c)=S_q(\boldsymbol{k}^{\dagger};c)$ 
by comparing the coefficients of $x^c$ in \eqref{x-duality}.
\end{proof}

\begin{rem}
When $q=1$ and $\boldsymbol{k}=(\{1\}^{r-1},2)$, our generating series 
$Z_q(\boldsymbol{k};x)$ becomes 
\[Z_1(\{1\}^{r-1},2;x)=\sum_{0<m_1<\cdots<m_r}
\frac{1}{(m_1-x)\cdots(m_r-x)m_r}. \]
Essentially the same series appears in the beginning of 
Granville's proof of the sum formula \cite{G}. 
We note that, however, our argument specialized to this case 
gives another proof of the sum formula: 
\begin{align*}
Z_1(\{1\}^{r-1},2;x)&=Z_1(\{1\}^{r-1},2;\varnothing;x)
=Z_1(\{1\}^{r};1;x)=\cdots\\
&=Z_1(\{1\}^{r-j};j+1;x)=\cdots=Z_1(\varnothing;r+1;x)=Z_1(r+1;x). 
\end{align*}
Here the connected sums $Z_1(\{1\}^{r-j};j+1;x)$ are the new ingredients. 
\end{rem}

\section*{Acknowledgement}
The authors sincerely thank the referee for helpful comments to improve the exposition.


\begin{thebibliography}{9}	
\bibitem{B}
D.~M.~Bradley, Multiple $q$-zeta values, 
\textit{J.~Algebra} \textbf{283} (2005), no.~2, 752--798.
\bibitem{G}
A.~Granville, 
A decomposition of Riemann's zeta function, 
in \textit{Analytic Number Theory}, London Mathematical Society Lecture Note Series 
Vol.~247, Cambridge University Press, Cambridge, 1997, 95--101. 
\bibitem{H}
M.~Hoffman, Multiple harmonic series, 
\textit{Pacific J.~Math.}, \textbf{152} (1992), no.~2, 275--290.
\bibitem{O}
Y.~Ohno, A generalization of the duality and sum formulas on the multiple zeta values, 
\textit{J.~Number Theory} \textbf{74} (1999), no.~1, 39--43.
\bibitem{Z}
D.~Zagier, Values of zeta functions and their applications,
\textit{First European Congress of Mathematics}, Vol.~II (Paris, 1992),
Progr.~Math., \textbf{120}, Birkh\"auser, Basel, 1994, 497--512.
\end{thebibliography}
\end{document}